\documentclass[a4paper,11pt]{article}
\usepackage{graphics,color}
\usepackage{amsmath}
\usepackage{amssymb}
\usepackage{mathrsfs}
\usepackage{cite}
\usepackage{verbatim}
\usepackage{float}
\usepackage{graphicx}
\usepackage{amsthm}
\usepackage{textcomp}
\usepackage{subfig}
\usepackage{esint}
\usepackage{enumerate}
\usepackage{hyperref}

\numberwithin{equation}{section}
\mathchardef\emptyset="001F

\newtheorem{Theorem}{Theorem}[section]
\newtheorem{Definition}[Theorem]{Definition}

\newtheorem{Corollary}[Theorem]{Corollary}
\newtheorem{Lemma}[Theorem]{Lemma}

\newcommand{\nada}[1]{}

\newcommand{\eps}{\varepsilon}
\newcommand{\grad}{\nabla}

\newcommand{\mres}{\mathbin{\vrule height 1.6ex depth 0pt width
0.13ex\vrule height 0.13ex depth 0pt width 1.3ex}}
\newcommand{\N}{\numberset{N}} 

\newcommand{\numberset}{\mathbb}
\newcommand{\Om}{\Omega} 
\newcommand{\R}{\numberset{R}}

\theoremstyle{definition}
\newtheorem{Remark}[Theorem]{Remark}

 \title{An $L^\infty$-variational problem involving the Fractional Laplacian
}
\author{
Simone Carano and Roger Moser}
\date{}

\begin{document}
\maketitle

\begin{abstract} 
For $s\in(0,1)$ and an open bounded set $\Om\subset\R^n$, we prove existence and uniqueness of absolute minimisers of the supremal functional 
$$E_\infty(u)=\|(-\Delta)^s u\|_{L^\infty(\R^n)},$$
where $(-\Delta)^s$ is the Fractional Laplacian of order $s$ and $u$ has prescribed Dirichlet data in the complement of $\Om$. We further show that the minimiser $u_\infty$ satisfies the (fractional) PDE 
$$
(-\Delta)^s u_\infty=E_\infty(u_\infty)\,\mathrm{sgn}f_\infty \qquad\mbox{in }\Om,
$$
for some analytic function $f_\infty\in L^1(\Om)$ obtained as the restriction of an $s$-harmonic measure $\mu$ in $\Om$. 
\end{abstract}

\section{Introduction}\label{sec:introduction}

For $n\in\N$ and $s\in(0,1)$, consider the supremal functional
\begin{align}\label{E infty}
E_{\infty}(u):=\|(-\Delta)^su\|_{L^\infty(\R^n)},
\end{align}
where $u$ belongs to the Fréchet potential space
$$
\mathcal{W}^{2s,\infty}(\R^n):=\bigcap_{1<p<\infty}\Big\{u\in W^{s,p}(\R^n):(-\Delta)^su\in L^\infty(\R^n)\Big\}.
$$
In the above, $(-\Delta)^su$ is the $s$-Laplacian of $u$, which is defined by
$$
(-\Delta)^s u(x):=c_{n,s}\,\int_{\R^n}\frac{u(x)-u(y)}{|x-y|^{n+2s}}dy\qquad\forall x\in\R^n,
$$
for $u$ smooth in $\R^n$. 
The previous integral is understood in the Cauchy principal value sense and $c_{n,s}:=\frac{s2^{2s\Gamma(\frac n2+s)}}{\pi^{n/2}\Gamma(1-s)}$, where $\Gamma$ denotes the Euler Gamma function. With this choice, the operator $(-\Delta)^s$ approaches the classical Laplacian as $s\to1^-$.\\
For $u\in W^{s,2}(\R^n)$, $(-\Delta)^su$ is understood in the usual weak sense as an element of $W^{-s,2}(\R^n)$. Thus, when $u\in \mathcal{W}^{2s,\infty}(\R^n)$, we are assuming that $(-\Delta)^su$ is represented by an $L^\infty$-function.\\
We aim to show existence and uniqueness of global minimisers of \eqref{E infty} given Dirichlet conditions on the complement of a fixed bounded open subset $\Om\subset\R^n$. Explicitely, the problem makes sense for exterior data $u_0\in\mathcal{W}^{2s,\infty}(\R^n)$ and the competitor class
$$
\mathcal{W}^{2s,\infty}_{u_0}(\Om):=u_0+\mathcal{W}^{2s,\infty}_0(\Om)
$$
where $\mathcal{W}^{2s,\infty}_0(\Om):=\{u\in\mathcal{W}^{2s,\infty}(\R^n): u=0 \,\mbox{ in } \R^n\setminus\Om\}$, leading to the minimum problem
$$
\min_{u\in \mathcal{W}^{2s,\infty}_{u_0}(\Om)} E_\infty(u)=\min_{u\in\mathcal{W}^{2s,\infty}_{u_0}(\Om)}\|(-\Delta)^su\|_{L^\infty(\R^n)}.
$$

 In addition, we will show that the minimiser satisfies a certain (fractional) PDE. The non-local nature of the operator makes it necessary to consider the behaviour of the minimiser on the whole $\R^n$ instead of $\Omega$. This is in contrast with the case when $s=1$, corresponding to the classical Laplacian, studied in \cite{KM}, where the sup norm is taken on the domain $\Om$. Another important difference with the local case is that no assumptions are required on the regularity of the boundary of $\Om$.  \\
\indent
Before entering into the details of our main result, we briefly describe the mathematical context of our problem. We are dealing with an $L^\infty$-variational problem of fractional differential order $2s$. The state of the art for integer order operators is well established: for the first order, most of the challenges have been succesfully addressed (see \cite{A} for the seminal work of Aronsson and \cite{Kbook} for a survey reference); the higher order case, recently started in \cite{KM,KP2}, is largely in development and still presents many open questions, since most of the approaches used in the first order case do not generalise. Without any pretension of being exhaustive, we refer also to  \cite{DK,KM2,KM3}  for a glimpse of the literature on higher order problems. On the other hand, the fractional order case seems to be largely unexplored. In \cite{CLM} (see also \cite{BCC}), the authors study a notion of infinity Fractional Laplacian, generalising the Aronsson equation to the fractional setting. This PDE is obtained as a limit as $p\to\infty$ of the Euler-Lagrange equation associated to the minimisation problem for the $W^{s,p}$-norm. More broadly, the asymptotic behaviour as $p\to\infty$ of relevant energies of functions in $W^{s,p}$ has been extensively studied in the literature (see e.g. \cite{BPS,AMRT, SR, LL}). However, the PDEs arising in these works are not related (at least in an obvious way) to any $L^\infty$-Dirichlet problem of fractional order. For this reason, our approach proceeds in the opposite direction: we begin by considering an $L^\infty$
 variational problem of fractional order and aim to derive a fractional PDE that characterises its (unique) minimiser. Supremal functionals involving nonlocal operators have been studied in \cite{Sc}, where the author establishes existence of minimisers for supremal functionals depending on the Riesz fractional gradient, without addressing uniqueness or deriving associated Euler–Lagrange-type equations. In this sense, our work can be viewed as a novel contribution to this framework, as it provides both uniqueness and characterisation of the minimiser via the corresponding governing fractional PDE.\\
\indent
Now we proceed to specify the setting of our problem. 
We assume that the prescribed data $u_0$ belong to $C_c^{2s+\gamma}(\R^n)$\footnote{We refer to Section \ref{sec:not} for the precise definition of the H\"older spaces $C^\alpha$, $\alpha\in\R^+$.}, for some $\gamma>0$. We could assume milder decay at infinity, however, the aim of this paper is to apply the methods of $L^\infty$-Calculus to the fractional setting, to give further evidence of the robustness of the techniques. For this purpose, in order to keep the technicality as light as possible, we keep the $C^{2s+\gamma}_c$-assumption. Notice that such $u_0$ belongs to $\mathcal{W}^{2s,\infty}(\R^n)$. Indeed, we have $(-\Delta)^s u_0\in C^{\gamma}_0(\R^n)$, with polynomial decay at infinity (see for instance \cite[Lemma 3.5]{Bo} for decay estimates and \cite[Prop. 2.6]{Si} for H\"older estimates). Also, from the definition of the Gagliardo seminorm (see Section \ref{sec:not}), it is easy to see that if $\sigma>s$, then $C^{\sigma}_c(\R^n)\subset W^{s,p}(\R^n)$ for all $p>1$. \\
On the other hand, by \cite[Prop. 2.9]{Si} and Morrey's inequality \cite[Thereom 8.2]{DPV}, we have $\mathcal{W}^{2s,\infty}(\R^n)\subset C^\alpha(\R^n)$ for every $\alpha<2s$. Thus, the regularity assumption on the data $u_0$ is not overly restrictive. \\
\indent
Now we can state our main result.
\begin{Theorem}\label{main thm}
Fix $s\in (0,1)$ and $n\in\N$, $n>2s$. Let $\Om\subset\R^n$ be an open bounded set and $E_\infty$ be the functional defined in \eqref{E infty}. Fix $u_0\in C^{2s+\gamma}_c(\R^n)$, for some $\gamma>0$, with $u_0\not\equiv0$ in $\R^n\setminus\Om$.
Then the problem
$$
e_\infty:=\inf_{\mathcal{W}_{u_0}^{2s,\infty}(\Om)}E_\infty
$$
is uniquely solvable, namely there exists a unique $u_\infty\in\mathcal{W}_{u_0}^{2s,\infty}(\Om)$ such that $E_\infty(u_\infty)=e_\infty$. 
In particular, $(-\Delta)^s u_\infty\in C^\gamma_{\mathrm{loc}}(\R^n\setminus\overline\Om)$ and $(-\Delta)^su_\infty(x)\to0$ as $|x|\to+\infty$.\\
Moreover, a fractional PDE can be derived as a necessary and sufficient condition for the minimality of $u_\infty$. Explicitely, there exists a measure $\mu\in\mathcal{M}(\R^n)$, $\mu\neq0$, with compact support and $|\mu|(\R^n)\leq1$, such that $\mu$ is $s$-harmonic in $\Om$ and 
\begin{align}\label{PDE u infty mu}
(-\Delta)^su_\infty=e_\infty\frac{d\mu}{d|\mu|}\qquad\mbox{in supp}|\mu|\setminus\partial\Om.
\end{align}
The identity above is understood between $L^\infty$-functions on supp$|\mu|\setminus\partial\Om$.\\
Moreover, the restriction $\mu\mres\Om$ is absolutely continuous w.r.t. the Lebesgue measure on $\Om$, i.e. $\mu\mres\Om=f_\infty\mathscr{L}^n\mres\Om$, for some function $f_\infty\in L^1(\Om)\setminus\{0\}$, which is real analytic in $\Om$. In particular, there holds
\begin{align}\label{PDE u infty}
(-\Delta)^su_\infty=e_\infty\mathrm{sgn}f_\infty\qquad\mbox{a.e.  in }\,\Om .
\end{align}
\end{Theorem}
In the above, a Radon measure $\mu$ in $\R^n$ with finite total variation is said to be $s$-harmonic on $\Om$ if
$$
\int_{\R^n}(-\Delta)^s \varphi \,d\mu=0,\qquad\forall\varphi\in C^\infty_c(\Om).
$$ 
Notice that the previous integral is well defined since $(-\Delta)^s\varphi\in C_0(\R^n)$, because $\varphi$ has compact support. \\
\indent
The non trivial choice of the ``boundary data" $u_0$ prevents that the variational problem trivialises, as we will show in the proof of Theorem \ref{main thm}.\\
\indent
 We prove existence of minimisers via $L^p$-approximation (i.e. Gamma Convergence), which is a very common technique in this context. The main idea is to consider and solve a variational problem for the (suitably normalised) $p$-norm of the Fractional Laplacian instead of the functional \eqref{E infty}. Once the existence of a minimiser $u_p$ of this ``$p$-version" of the problem is established, we recover the original problem for $E_\infty$ in the limit as $p\to\infty$ and, thanks to compactness properties, we show that $u_p\to u_\infty$ in the appropriate convergence, where $u_\infty$ is a minimiser of $E_\infty$.  This proof works under the natural assumption that $u_0\in\mathcal{W}^{2s,\infty}(\R^n)$. On the other hand, the additional $C^{2s+\gamma}_c$-regularity assumption on $u_0$ is needed
in the derivation of the PDE in \eqref{PDE u infty mu}. Also in this case, we argue by $L^p$-approximation, precisely by passing to the limit as $p\to\infty$ in the Euler-Lagrange equation satisfied by $u_p$. The relevant function $f_p$ that appears in this PDE naturally inherits $s$-harmonicity properties and a uniform bound on the mass, as consequences of the minimality of $u_p$. This gives rise to the $s$-harmonic measure $\mu$, as limit of $f_p$ for $p\to\infty$. At this point, a technical argument based on the $C^{2s+\gamma}_c$-regularity of $u_0$ ensures that $\mu$ is represented by a non-zero analytic function on $\Om$.
\\
\indent
The uniqueness of $u_\infty$ is a consequence of the fact that it satisfies \eqref{PDE u infty}. To be more precise, the crucial point in proving uniqueness is to guarantee that \textit{any} minimiser $u$ satisfies $|(-\Delta)^su|=e_\infty$ in $\Omega$. This is suprisingly in complete analogy with the local case (see \cite{KM,CKM}), despite the non-local nature of the operator $(-\Delta)^s$. In fact, the argument of the proof is similar to the one in \cite{CKM}, inspired in turn to a technique in \cite{M}, and consists in the insertion of a penalisation term in the $L^p$-approximation problem, forcing the approximating minimiser to annul the penalisation and recover the preselected minimiser of the $L^\infty$-problem in the limit. Moreover, in general, the uniqueness property of a global minimiser ensures that it is also an \textit{absolute minimiser} (see for instance \cite[Corollary 1.4]{CKM}).\\
\indent
The behaviour of the measure $\mu$ on $\partial\Om$ remains an open question. However, we believe that, as for the local case \cite{KM}, it should be possible to rule out concentration phenomena at the boundary by constructing suitable test functions for the Euler-Lagrange equation satisfied by $u_p$. Such a result may require regularity assumptions on the boundary. \\
\indent
Finally, we observe that if we restrict equation \eqref{PDE u infty mu} to $\R^n\setminus\overline\Om$, we deduce that, thanks to the continuity of $(-\Delta)^su_\infty$, the measure $\mu$ must have constant sign on each connected component of supp$|\mu|$. However, at the present stage, we are not able to ensure the non-emptyness of supp$|\mu|$ on $\R^n\setminus\overline\Om$. More in general, to our best knowlegde, behaviour and regularity properties of $s$-harmonic measures (outside their $s$-harmonicity domain) are largely unexplored, motivating further investigation.

\section{Notation}\label{sec:not}
In the whole paper, $s\in(0,1)$ and $n\in\N$, $n>2s$, are fixed numbers. \\
For $k\in\N\cup\{\infty\}$, we denote by $C_c^k(\R^n)$ and $C_0^k(\R^n)$ respectively the space of $C^k$-functions with compact support and the space of $C^k$-functions vanishing at infinity with all their derivatives up to order $k$. If $k=0$, we write simply $C_c(\R^n)$ and $C_0(\R^n)$.\\
 Let $U\subseteq\R^n$ be an open set and $\alpha\in(0,1)$. We denote by $C^{\alpha}(U)$ the space of H\"older continuous functions of order $\alpha$ in $U$. By $C^{\alpha}_{\mathrm{loc}}(U)$ we refer to the space of locally H\"older continuous functions of order $\alpha$ in $U$, i.e. we say that $u\in C^{\alpha}_{\mathrm{loc}}(U)$ if for every $x\in U$ there exists a neighborhood $B$ of $x$ such that $u\in C^{\alpha}(B)$. More in general, if $\alpha\in\R^+\setminus\N$, we denote by $C^\alpha(U)$ (resp. $C_{\mathrm{loc}}^\alpha(U)$ ) the space of $\lfloor\alpha\rfloor$ times (resp. locally) differentiable functions with derivatives of order $\lfloor\alpha\rfloor$ being (resp. locally) H\"older continuous of order $\alpha-\lfloor\alpha\rfloor$.\\
Suppose that $\gamma>0$ is such that $2s+\gamma$ is not an integer and $u\in C^{2s+\gamma}_{\mathrm{loc}}(\R^n)\cap L^1_s(\R^n)$. 
The space $L^1_s(\R^n)$ is the set of Lebesgue measurable functions $u:\R^n\to\R$ with 
$$
\int_{\R^n}\frac{|u(x)|}{1+|x|^{n+2s}}dx<\infty.
$$
Then, the $s$-Laplacian of $u$ is pointwise defined at every $x\in\R^n$ as
$$
(-\Delta)^s u(x):=c_{n,s}\,\mathrm{P.V.}\,\int_{\R^n}\frac{u(x)-u(y)}{|x-y|^{n+2s}}dy:=c_{n,s}\lim_{\eps\to0}\,\int_{\R^n\setminus B_\eps(x)}\frac{u(x)-u(y)}{|x-y|^{n+2s}}dy,
$$
where
$c_{n,s}:=\frac{s2^{2s\Gamma(\frac n2+s)}}{\pi^{n/2}\Gamma(1-s)}$ and $\Gamma$ denotes the Euler Gamma function.\\
For $1\leq p<\infty$ and an open set $U\subseteq\R^n$, we define the fractional Sobolev space $W^{s,p}(U)$ as
$$
W^{s,p}(U)=\left\{u\in L^p(U): [u]_{W^{s,p}(U)}<\infty\right\},
$$
where $[\cdot]_{W^{s,p}(U)}$ stands for the Gagliardo seminorm, namely
$$
[u]_{W^{s,p}(U)}:=\left(\int_U dx\int_U\frac{|u(x)-u(y)|^p}{|x-y|^{n+sp}}dy\right)^{\frac1p}.
$$
Also, $W^{s,p}(U)$ is a Banach space, with the norm
$$
\|u\|_{W^{s,p}(U)}:=\|u\|_{L^p(U)}+[u]_{W^{s,p}(U)}.
$$
When $p=2$ we use the notation $H^s(U)$ to refer to the Hilbert space $W^{s,2}(U)$.\\
For $u\in H^s(\R^n)$, we can define $(-\Delta)^su$ in the weak sense as an element of the dual space $H^{-s}(\R^n)$:
$$
\langle(-\Delta)^su,v\rangle:=\frac{c_{n,s}}2\int_{\R^n\times\R^n}\frac{(u(x)-u(x))(v(x)-v(y))}{|x-y|^{n+2s}}dxdy\qquad\forall v\in H^s(\R^n).
$$
For $\infty\geq p\geq (2_s)_*:= \frac{2n}{n+2s}$, we write $(-\Delta)^su\in L^p(\R^n)$ if there exists $g\in L^p(\R^n)$ such that
$$
\langle(-\Delta)^su,v\rangle=\int_{\R^n} gv\qquad\forall v\in H^s(\R^n),
$$
where we used that $H^s(\R^n)\subset L^p(\R^n)$, for all $2\leq p\leq 2_s^*=\frac{2n}{n-2s}$.\\
If $2s>1$, write $2s=1+\sigma$. Then, for $p\in[1,\infty)$, we define $$W^{2s,p}(U):=\{u\in W^{1,p}(U): Du\in W^{\sigma,p}(U)\},$$ equipped with the corresponding suitable norm $\|\cdot\|_{W^{2s,p}(U)}$.\\
Finally, for a bounded open set $\Omega\subset\R^n$ and $p\in (1,\infty)$, we denote
$$
W_0^{s,p}(\overline \Om):=\{u\in W^{s,p}(\R^n): u=0 \mbox{ in }\R^n\setminus\Om\}.
$$
In particular, we write $H_0^s(\overline\Om):=W^{s,2}_0(\overline\Om)$. If $\Omega$ has $C^2$-boundary, then $H^s_0(\overline\Om)$ coincide with the closure of $C^\infty_c(\Om)$ under the norm $\|\cdot\|_{H^s(\R^n)}$.

\section{Properties of weak solutions to the fractional Dirichlet problem}\label{prel}
In this section, we recall some useful results about weak and very weak solutions to the Dirichlet problem for the Fractional Laplacian. We start by recalling the notion of weak solution.
\begin{Definition}\label{def weak sol}
For a bounded open set $\Om\subset\R^n$ and $s\in(0,1)$, let $f\in H^{-s}(\overline\Om):=(H_0^s(\overline\Om))^*$. A function $u\in H^s_0(\overline\Om)$ is a weak solution to
\begin{equation}\label{dir}
\left\{
\begin{aligned}
(-\Delta)^s u&=f\qquad&&\mbox{in }\Om,\\
u&=0&&\mbox{in }\R^n\setminus\Om,
\end{aligned}
\right.
\end{equation}
if for every $v\in H^s_0(\overline\Om)$ it holds
$$
\frac{c_{n,s}}2\int_{\R^n\times\R^n}\frac{(u(x)-u(x))(v(x)-v(y))}{|x-y|^{n+2s}}dxdy=\langle f,v \rangle_{H^{-s}(\overline\Om),H^s_0(\overline\Om)},
$$
where $\langle \cdot,\cdot \rangle_{H^{-s}(\overline\Om),H^s_0(\overline\Om)}$ stands for the inner product in the duality between $H^s_0(\overline\Om)$ and $H^{-s}(\overline\Om)$.
\end{Definition}
Existence and uniqueness of weak solutions to \eqref{dir} are simple consequence of Lax-Milgram's Theorem (see \cite[Prop. 2.1]{BWZ}). The unique weak solution to \eqref{dir} enjoys local $L^p$-regularity: if $f\in L^p(\Om)$, $p\geq2$, then $u\in W^{2s,p}_{\mathrm{loc}}(\Om)$ (see e.g. \cite[Theorem 1.3]{BWZ}). Moreover, global $L^p$-estimates of Calder\'on-Zygmund type are available, as recently proved in \cite[Corollary 1.7]{AFLY}, which we state in the following, more particular, version\footnote{The notion of weak solution of Definition \ref{def weak sol} is more stronger (and then compatible) with the notion of weak solution considered in \cite{AFLY} (at least for $C^2$-domain, see e.g. \cite[Def. 1.2]{SP}).}.
\begin{Theorem}[Fractional Calder\'on-Zygmund estimates]\label{CZ thm}
Let $\Om\subset\R^n$ be a bounded domain with boundary of class $C^2$ and $s\in(0,1)$. Let $f\in L^m(\Om)$ and $u$ be the unique (weak) solution to \eqref{dir}. Then, we have:
\begin{itemize}
\item if $1\leq m\leq \frac ns$, then for all $1<p<m^*_s:=\frac{nm}{n-ms}$ there exists $C>0$ such that
$$
\|u\|_{W^{s,p}(\R^n)}\leq C\|f\|_{L^m(\Om)};
$$
\item if $m>\frac ns$, then for all $1<p<\infty$ there exists $C>0$ such that
$$
\|u\|_{W^{s,p}(\R^n)}\leq C\|f\|_{L^m(\Om)}.
$$
\end{itemize}
Here, $C > 0$ are positive constants depending only on $n, s, p, m$ and $\Om$.
\end{Theorem}
\begin{Remark}
We observe that, since $m^*_s>m$, we can choose $p=m$ in Theorem \ref{CZ thm}, obtaining an estimate of the type $W^{s,p}(\R^n)-L^p(\Om)$ for all $p\in (1,\infty)$:
$$
\|u\|_{W^{s,p}(\R^n)}\leq C\|f\|_{L^p(\Om)},
$$
for a constant $C=C(n,s,p,\Om)>0$.
Although always refering to Theorem \ref{CZ thm}, we will use only this estimate throughout the whole paper.
\end{Remark}

Now, we introduce the notion of \textit{very weak solution} to the Fractional Laplace problem. First, we define the space
$$
L^\infty_s(\R^n):=\{u\in L^\infty(\R^n): \|(1+|\cdot|^{n+2s})|u|\|_{L^\infty(\R^n)}<\infty\}.
$$
\begin{Definition}\label{very weak def}
Let $\Om\subset\R^n$ be open and $g\in L^1_s(\R^n)$. A function $u\in L^1_s(\R^n)$ is a very weak solution to
\begin{equation}\label{lap}
\left\{
\begin{aligned}
(-\Delta)^s u&=0\qquad&&\mbox{in }\Om,\\
u&=g&&\mbox{in }\R^n\setminus\Om,
\end{aligned}
\right.
\end{equation}
if for every $\varphi$ with compact support in $\Om$, such that $(-\Delta)^s\varphi\in L^\infty_s(\R^n)$, it holds
\begin{equation}\label{very weak}
\left\{
\begin{aligned}
&\int_{\R^n}u(-\Delta)^s \varphi=0,\\
&u=g\qquad\mbox{in }\R^n\setminus\Om.
\end{aligned}
\right.
\end{equation}
\end{Definition}

\begin{Remark}\label{rem on class}
In the definition above, the regularity of the tests $\varphi$ is tacitly assumed to be $H^s(\R^n)$. 
 In the literature, it is sometimes required that $\varphi$ has $L^\infty(\R^n)$ or $C^s(\R^n)$ regularity. However, since $(-\Delta)^s\varphi\in L^\infty_s(\R^n)$, all these assumptions amount to the same class, which are the functions in $W^{2s,p}_0(\overline\Om)$, for all $p\in [2,\infty)$, with $(-\Delta)^s\varphi\in L^\infty_s(\R^n)$.\\ Indeed, assume $\varphi\in L^2(\R^n)$, with $\varphi=0$ in $\R^n\setminus\Om$ and (distributional) $s$-Laplacian $(-\Delta)^s\varphi\in L^\infty_s(\R^n)$. Then, since $\varphi=0$ outside of $\Om$, one has $(-\Delta)^s\varphi$ is smooth outside of $\Omega$, decaying as $|x|^{-n-2s}$ at infinity. Therefore, $(-\Delta)^s\varphi\in L^1(\R^n)$ as well, so that $(-\Delta)^s\varphi\in L^p(\R^n)$ for all $p\in[1,\infty]$. Now recall that (see e.g. \cite[Corollary 4.56]{ADV}) $$\{\varphi\in L^2(\R^n): (-\Delta)^s \varphi\in L^2(\R^n)\}\simeq W^{2s,2}(\R^n).$$ So, in particular, $\varphi\in H^s(\R^n)$ and $(-\Delta)^s \varphi$ is well defined also in the weak sense. By \cite[Theorem 2.4]{BWZ}), we conclude that
$\varphi\in W^{2s,p}(\R^n)$ for all $p\in [2,\infty)$.
\end{Remark}
The regularity of very weak solutions to \eqref{lap} can be improved locally to Sobolev and even to classical regularity. This result is contained in \cite{CCLP}, where the authors prove it on the unit ball of $\R^n$. Due to the local nature of the estimates, we may state the result for any open set $\Om\subset\R^n$.
\begin{Theorem}[Regularity of very weak solutions]\label{thm CCPL}
Let $\Om\subset\R^n$ be an open set and $s\in(0,1)$. Let $g\in L^1_s(\R^n)$ and $u$ a very weak solution to \eqref{lap}. Then, we have:
\begin{itemize}
\item $u\in H^s_\mathrm{loc}(\Om)$ and
$$
\|u\|_{H^s(B)}\leq C(B)\|u\|_{L^1_s(\R^n)}\qquad\forall B\Subset\Om;
$$
\item $u$ is real analytic in $\Om$ and
$$
\|D^\alpha u\|_{L^\infty(B')}\leq c^{|\alpha|}\alpha!C(B,B',n,s)\left(\|u\|_{L^\infty(B')}+\|u\|_{L^1_s(\R^n)}\right)\qquad\forall \alpha\in \N^n_0\quad\forall B'\Subset B\Subset\Om,
$$
for a universal constant $c>0$.
\end{itemize}
\end{Theorem}

\section{Proof of the main results}\label{sec:proofs}
We divide the proof of Theorem \ref{main thm} into four parts: We start with ensuring the positivity of $e_\infty$, then we show existence, PDE derivation, and uniqueness of minimisers respectively.
\subsection{Strict positivity of the minimum}\label{subsec strict}
Let us prove that under the assumption that $u_0\not\equiv0$ in $\R^n\setminus\Om$, we have $e_\infty>0$.\\ By seek of contradiction, assume that $e_\infty=0$. Then, there exists a sequence $(u_k)_k\subset\mathcal{W}^{2s,\infty}_{u_0}(\Om)$ with $E_\infty(u_k)\to 0$ as $k\to\infty$. Thus, 
\begin{align}\label{fk}
(-\Delta)^su_k\to0\qquad\mbox{ uniformly in } \R^n.
\end{align}
Set $v_k:=u_k-u_0$, then $v_k\in W^{s,p}_0(\overline\Om)$ and $f_k:=(-\Delta)^sv_k=(-\Delta)^su_k-(-\Delta)^su_0\in L^p_{\mathrm{loc}}(\R^n)$, for all $p>1$. In particular, if $\Om'\subset\R^n$ is an open bounded set with $C^2$-boundary such that $\Om\Subset\Om'$, we have $f_k\in L^p(\Om')$. Now, observe that $v_k$ is a weak solution to
\begin{equation}\nonumber
\left\{
\begin{aligned}
(-\Delta)^sv_k&=f_k\qquad&&\mbox{in }\Om',\\
v_k&=0 &&\mbox{in }\R^n\setminus\Om'.
\end{aligned}
\right.
\end{equation}
 By the Calder\'on-Zygmund estimates in Theorem \ref{CZ thm}, we have
\begin{align*}
\|v_k\|_{W^{s,p}(\R^n)}&\leq C\|f_k\|_{L^p(\Om')}\leq C\left(\|(-\Delta)^su_k\|_{L^p(\Om')}+\|(-\Delta)^su_0\|_{L^p(\Om')}\right)
\end{align*}
for a constant $C>0$ independent of $k$. From these estimates and \eqref{fk}, we obtain 
$$
\limsup_{k\to\infty}\|u_k\|_{W^{s,p}(\R^n)}\leq \|u_0\|_{W^{s,p}(\R^n)}+C\|(-\Delta)^su_0\|_{L^p(\Om')}.
$$
 Thus, $(u_k)_k$ is bounded in $W^{s,p}(\R^n)$, for all $p>1$. By the weak compactness of $W^{s,p}(\R^n)$, there exists $u\in W^{s,p}(\R^n)$ and a (non relabelled) subsequence of $u_k$ such that $u_k\rightharpoonup u$ weakly in $W^{s,p}(\R^n)$, as $k\to\infty$. Clearly, $u=u_0$ in $\R^n\setminus\Om$ and $(-\Delta)^s u=0$ in $\R^n$, therefore $u\in \mathcal{W}^{2s,\infty}_{u_0}(\Om)$.
But, since an entire $s$-harmonic function is necessarily affine (see \cite{F}) and $u\in L^p(\R^n)$ for every $p>1$, we must have $u=0$ in $\R^n$, that is a contradiction.\\

\subsection{Existence}\label{subs:existence}
Let $p>1$ and recall that $u_0\in\mathcal{W}^{2s,\infty}(\Om)$. Let $w\in C^2(\R^n)\cap C_0(\R^n)$ be such that $w>0$ and $\int_{\R^n}w=1$. Define the weighted potential space
$$
\mathscr{L}^{2s,p}_w(\R^n):=\left\{u\in W^{s,p}(\R^n):\int_{\R^n}\left|(-\Delta)^su\right|^pw<\infty\right\},
$$ 
equipped with the norm
$$
\|u\|_{\mathscr{L}^{2s,p}_w(\R^n)}:=\|u\|_{W^{s,p}(\R^n)}+\left(\int_{\R^n}\left|(-\Delta)^su\right|^pw\right)^{\frac1p}.
$$
This is a reflexive Banach space, since it is isometrically isomorphic to a closed subspace of $W^{s,p}(\R^n)\times L^p(\R^n)$. We want to consider its subspace
$$
\mathscr{L}^{2s,p}_{w,u_0}(\Om):=\{u\in\mathscr{L}^{2s,p}_w(\R^n):u=u_0\quad \mbox{in }\R^n\setminus\Om\}
$$
and minimise the functional
\begin{align}\label{Ep}
E_p(u):=\left(\int_{\R^n}\left|(-\Delta)^su(x)\right|^pw(x)dx\right)^{\frac{1}{p}}
\end{align}
on $\mathscr{L}^{2s,p}_{w,u_0}(\Om)$.
Notice that ${\mathcal W}^{2s,\infty}(\R^n)\subset\mathscr{L}^{2s,p}_{w}(\R^n)$, so, in particular, we have $E_p(u_0)<\infty.$\\
Let us prove that $E_p$ is (weakly) coercive in $\mathscr{L}^{2s,p}_{w,u_0}(\Om)$. Assume that $E_p(u)\leq M$ for some $u\in \mathscr{L}^{2s,p}_{w,u_0}(\Om)$. 
Set $v:=u-u_0\in W^{s,p}_0(\overline\Om)$ and $f:=(-\Delta)^sv=(-\Delta)^su-(-\Delta)^su_0\in L^p_{\mathrm{loc}}(\R^n)$, for all $p>1$. In particular, if $\Om'\subset\R^n$ is an open bounded set with $C^2$-boundary such that $\Om\Subset\Om'$, we have $f\in L^p(\Om')$. Thus, the function $v$ (weakly) solves the Dirichlet problem
\begin{equation}\nonumber
\left\{
\begin{aligned}
(-\Delta)^sv&=f\qquad&&\mbox{in }\Om',\\
v&=0 &&\mbox{in }\R^n\setminus\Om'.
\end{aligned}
\right.
\end{equation}
Applying the Calder\'on-Zygmund estimates in Theorem \ref{CZ thm}, we have
$$
\|v\|_{W^{s,p}(\R^n)}\leq C\|f\|_{L^p(\Om')},
$$
for some constant $C=C(\Om',s,n,p)$. From this estimate and the trivial inequality
$$
\|g\|_{L^p(\Om')}\leq{\big(\inf_{\Om'}w\big)}^{-\frac1p}\left(\int_{\R^n}\left|g\right|^pw\right)^{\frac{1}{p}}\qquad\forall g\in L^\infty(\R^n),
$$
we deduce
\begin{align*}
\|u\|_{W^{s,p}(\R^n)}&\leq C\left(\|(-\Delta)^s u\|_{L^p(\Om')}+\|(-\Delta)^su_0\|_{L^p(\Om')}+\|u_0\|_{W^{s,p}(\R^n)}\right)\\&\leq C\left(E_p(u)+\|u_0\|_{\mathscr{L}^{2s,p}_w(\R^n)}\right)
\\&\leq C\left(M+\|u_0\|_{\mathscr{L}^{2s,p}_w(\R^n)}\right),
\end{align*}
for a constant $C=C(\Om',s,n,p,w)$.
The right hand side provides an upper bound also for $\|u\|_{\mathscr{L}^{2s,p}_w(\R^n)}$, ensuring the coercivity of $E_p$.\\
Clearly, $E_p$ is weakly lower semicontinuous in $\mathscr{L}^{2s,p}_w(\R^n)$. 
Therefore, we can apply Direct Methods on $\mathscr{L}^{2s,p}_{w,u_0}(\Om)$ to get a minimiser $u_p$ of $E_p$ in this space.\\
Now fix $q>1$ and consider the sequence of (extended) functionals $(E_p)_p$ in $\mathscr{L}^{2s,q}_{w,u_0}(\Om)$ defined as
\begin{equation}\nonumber
E_p(u)=\left\{
\begin{aligned}
&\left(\int_{\R^n}|(-\Delta)^su(x)|^pw(x)\right)^{\frac{1}{p}}\!\!\!\!\!&,\quad&u\in \mathscr{L}^{2s,q}_{w,u_0}(\Om)\cap\mathscr{L}^{2s,p}_{w,u_0}(\Om),\\
&+\infty&,\quad&u\in \mathscr{L}^{2s,q}_{w,u_0}(\Om)\setminus \mathscr{L}^{2s,p}_{w,u_0}(\Om),
\qquad\end{aligned}
\right.
\begin{aligned}
\forall p>1.
\end{aligned}
\end{equation}
In a similar way, we can define $E_\infty$ in $\mathscr{L}^{2s,q}_{w,u_0}(\Om)$ as
\begin{equation}\nonumber
E_\infty(u)=\left\{
\begin{aligned}
&\|(-\Delta)^su\|_{L^\infty(\R^n)}\!\!\!\!\!&,\quad&u\in\mathcal{W}_{u_0}^{2s,\infty}(\Om),\\
&+\infty&,\quad&u\in \mathscr{L}^{2s,q}_{w,u_0}(\Om)\setminus\mathcal{W}_{u_0}^{2s,\infty}(\Om).
\end{aligned}
\right.
\end{equation}
  Using the H\"older inequality, it is not difficult to prove that $(E_p)_p$ is monotone increasing and
$$
\lim_{p\to\infty}E_p(u)=E_\infty(u)\qquad\forall u\in\mathscr{L}^{2s,q}_{w,u_0}(\Om).
$$
Let us show that $(E_p)_p$ is equi-coercive in $\mathscr{L}^{2s,q}_{w,u_0}(\Om)$ with respect to $p$. Assume $E_p(u)\leq M$ for all $p\geq q$. By fixing a domain $\Om'\Supset\Om$ with $C^2$-boundary, we can apply the fractional Calderon-Zygmund estimates as before. Together with the H\"older inequality, this yields
\begin{align*}
\|u\|_{\mathscr{L}^{2s,q}_{w}(\R^n)}&\leq C\left(\|u_0\|_{\mathscr{L}^{2s,q}_w(\R^n)}+E_q(u)\right)\\
&\leq C\left(\|u_0\|_{\mathscr{L}^{2s,q}_w(\R^n)}+E_p(u)\right)\\
&\leq C\left(\|u_0\|_{\mathscr{L}^{2s,q}_w(\R^n)}+M\right),
\end{align*}
for a constant $C=C(\Om',s,n,q,w)$,
giving the equi-coercivity of $(E_p)_p$. \\
Since $(E_p)_p$ is a monotone increasing sequence of weakly lower semicontinuous functionals
 on $\mathscr{L}^{2s,q}_{w,u_0}(\Om)$, its $\Gamma$-limit is given by (see \cite[Remark 2.12]{Br})
$$
\Gamma-\lim_p E_p=\lim_pE_p=E_\infty.
$$
By the Fundamental Theorem of Gamma Convergence \cite[Thm. 2.10]{Br}, if $(u_p)_p$ is a sequence of minimisers of $E_p$, up to passing to a subsequence, we have $u_p\rightharpoonup u_\infty$ weakly $\mathscr{L}^{2s,q}_{w,u_0}(\Om)$ where $u_\infty$ is a minimiser of $E_\infty$ on $\mathcal{W}_{u_0}^{2s,\infty}(\Om)$. By definition of $E_\infty$, we necessarily have $u_\infty\in\mathcal{W}_{u_0}^{2s,\infty}(\Om)$.
Furthermore, Theorem 2.10 in \cite{Br} also implies
\begin{align}\label{limit ep}
E_p(u_p)\to E_\infty({u_\infty}) \quad\mbox{as }p\to\infty.
\end{align}

\subsection{PDE derivation}\label{subs: PDE derivation}

Now we work towards the proof of \eqref{PDE u infty}.
Let us begin by writing down the Euler-Lagrange equation satisfied by $u_p$: 
\begin{align}\label{EL for Ep}
\int_{\R^n}f_p(x)(-\Delta)^sv(x)\, dx=0,\quad \forall v\in \mathscr{L}^{2s,p}_{w,0}(\Om),
\end{align}
where, for $e_p:=E_p(u_p)$, $f_p:\R^n\to\R$ is defined by
\begin{align}\label{fp}
f_p:=e_p^{1-p}w|(-\Delta)^su_p|^{p-2}(-\Delta)^su_p.
\end{align}
By the non-trivial choice of the boundary data $u_0$, we have $e_\infty>0$. Thus, we may assume without loss of generality that $e_p\geq e_1>0$ for all $p>1$. So, $f_p$ is a well defined measurable function on $\R^n$.\\
Let us prove that that $(f_p)_p$ is bounded in $L^1(\R^n)$: using the H\"older inequality w.r.t. the measure $w\mathscr{L}^n\mres\R^n$ and recalling that $\int_{\R^n}w=1$, we have
\begin{equation}\label{fp bdd}
\begin{aligned}
\int_{\R^n}|f_p|&=e_p^{1-p}\int_{\R^n}|(-\Delta)^su_p|^{p-1}w\\
&\leq e_p^{1-p}\left(\int_{\R^n}|(-\Delta)^su_p|^{p}w\right)^{\frac{p-1}{p}}\left(\int_{\R^n}w\right)^\frac1p\\
&=e_p^{1-p}e_p^{p-1}\cdot 1\\
&=1.
\end{aligned}
\end{equation}
Observe that by Remark \ref{rem on class}, a test function $\varphi$ in Definition \ref{very weak def} is such that $\varphi \in W^{2s,p}_0(\overline\Om)$. In particular, $\varphi\in\mathscr{L}^{2s,p}_{w,0}(\Om)$, so it is a good test function also for \eqref{EL for Ep}. Hence, $f_p$ is a {\it{very weak} $s$-harmonic} function on $\Om$. Equivalently, it is a very weak solution to the Fractional Laplace problem on $\Om$ w.r.t its own Dirichlet data. \\
Now we have all the tools to construct the function $f_\infty$ appearing in \eqref{PDE u infty}. Thanks to Theorem \ref{thm CCPL}, we have $f_p\in H^s_{\mathrm{loc}}(\Om)$ and the following estimate holds: for any open set $B\Subset\Om$ there exists a constant $C=C(B)>0$ such that
\begin{align}\label{Hs est}
\|f_p\|_{H^s(B)}\leq C\|f_p\|_{L^1(\R^n)}.
\end{align}
By \eqref{fp bdd} and compact Sobolev embeddings (see \cite[Corollary 7.2]{DPV}), if $n>2s$, there exist a subsequence of $(f_p)$, which we denote by $(f_p)$ as well, and a function $f_\infty$ such that
$$
f_p\to f_\infty \quad\mbox{in }L^q(B) \quad\forall q\in[1,2^*_s),
$$
where $2^*_s:=\frac{2n}{n-2s}$. Observe that, thanks to \eqref{fp bdd}, $f_\infty\in L^1(\Om)\cap L^q(B)$ for all $q\in[1,2^*_s)$ and up to passing to a further subsequence 
\begin{align}\label{ae conv}
f_p\to f_\infty \quad\mbox{a.e. in }\Om.
\end{align}
Actually, the convergence is locally uniform in $\Om$. Indeed, by classical regularity (Theorem \ref{thm CCPL}), we have 
$f_p\in C^\infty(\Om)$. Thus, $f_p$ is $s$-harmonic in $\Om$ in the classical sense. In particular, it is a local weak solution to the Fractional Laplace problem in $\Om$ (w.r.t. its own Dirichlet boundary data). So, from Theorems 1.4 and 3.2 in \cite{BLS}, for any open set $B\Subset\Om$, there exist $\sigma\in(0,1)$ and a constant $C=C(\sigma,n,s,B)$ such that 
$$
\|f_p\|_{C^{0,\sigma}(B)}\leq C\left(\|f_p\|_{L^2(B)}+\|f_p\|_{L^1(\R^n)}\right).
$$
Hence, by \eqref{Hs est} we obtain
\begin{align}\label{unif est}
\|f_p\|_{C^{0,\sigma}(B)}\leq C\|f_p\|_{L^1(\R^n)}.
\end{align}
Thus, the $L^1$-bound \eqref{fp bdd} of $f_p$ in $\R^n$ ensures
\begin{align}\label{unif conv}
f_p\to f_\infty\quad\mbox{locally uniformly in }\Om,
\end{align}
up to passing to a further subsequence.\\
Let us prove that $f_\infty$ is real analytic in $\Om$. By Theorem \ref{thm CCPL}, for any pair of open sets $B'\Subset B\Subset\Om$ and multindex $\alpha\in\N^n$, we have
$$
\|D^\alpha f_p\|_{L^\infty(B')}\leq c^{|\alpha|}\alpha!C(B,B',n,s)\left(\|f_p\|_{L^\infty(B)}+\|f_p\|_{L^1(\R^n)}\right).
$$
for a universal constant $c>0$.
Thus, using \eqref{unif est} and the $L^1$-bound of $f_p$, we can pass to the limit in the above estimate, obtaining
$$
\|D^\alpha f_\infty\|_{L^\infty(B')}\leq c^{|\alpha|}\alpha!C(B,B',n,s,\sigma),
$$
which provides the analyticity of $f_\infty$ in $\Om$.\\
Moreover, again from \eqref{fp bdd}, we infer the existence of a Radon measure $\mu\in\mathcal{M}(\R^n)$ with $|\mu|(\R^n)\leq{1}$ and such that, up to passing to a subsequence,
$$
f_p\stackrel{*}{\rightharpoonup}\mu\quad\mbox{in }\mathcal{M}({\R^n}).
$$
In particular, by \eqref{EL for Ep}, $f_p$ is $s$-harmonic in the sense of distributions, namely
\begin{align}\label{EL distr fp}
\int_{\R^n}f_p(-\Delta)^s \varphi \,dx=0\qquad\forall\varphi\in C^\infty_c(\Om).
\end{align}
Notice that if $\varphi\in C^\infty_c(\Om)$, then $(-\Delta)^s\varphi\in C^\infty(\R^n)$ and 
$$
|(-\Delta)^s\varphi(x)|\leq c_{n,s}\|\varphi\|_\infty|\Om| d(x,\partial\Om)^{-n-2s}\qquad\forall x\in \R^n\setminus\Om.
$$
So, $(-\Delta)^s\varphi\in C_0(\R^n)$ and we can pass to the limit in \eqref{EL distr fp}, obtaining
$$
\int_{\R^n}(-\Delta)^s \varphi \,d\mu=0\qquad\forall\varphi\in C^\infty_c(\Om),
$$ 
i.e. $\mu$ is $s$-harmonic in $\Om$. In addition, thanks to \eqref{ae conv}, we can write
\begin{align}\label{mu on Om}
\mu\,\mres\Om=f_\infty\mathscr{L}^n\mres\Om.
\end{align}
The next result is central in our argument.

\begin{Lemma}{\bf (Non-triviality of $f_\infty$)}\label{non triv lemma}
The map $f_\infty$ constructed above is such that $f_\infty\not\equiv0$ in $\Om$.
\end{Lemma}
Before proving Lemma \ref{non triv lemma}, we need a couple of technical results. They can be seen essentially as a consequence of Theorem 1.1 in \cite{DSV}. The latter can be stated as ``all functions are locally $s$-harmonic up to a small error". We recall (and briefly prove) it here in a more general version, for H\"older continuous functions, which can be easily obtained from Theorem 1.1 in \cite{Kr}, which is another version of \cite[Theorem 1.1]{DSV}.
\begin{Theorem}{\bf (All functions are locally $s$-harmonic up to a small error)}\label{DSV}
Let $\alpha\in\R^+$ and suppose that $u\in C^\alpha(\overline B_1)$. Then, for every $\eps>0$, there exists $u_\eps\in C^\infty_c(\R^n)$ such that
\begin{equation}\nonumber\left\{
\begin{aligned}
&(-\Delta)^s u_\eps=0\quad\mbox{in }\overline B_1,\\
&\|u-u_\eps\|_{C^\alpha(B_1)}\leq\eps.
\end{aligned}
\right.
\end{equation}
\end{Theorem}
\begin{proof}
For integer $\alpha$, the result reduces to Theorem 1.1 in \cite{Kr}. So, we can assume $\alpha\in\R^+\setminus\N$.\\
Consider a mollification $u_\eta\in C^\infty(\R^n)$ of $u$, with $\eta>0$, such that
$$
\|u-u_\eta\|_{C^\alpha(B_1)}<\eta.
$$
Notice that $u_{{\eta}_{|\bar B_1}}\in C^k(\bar B_1)$, for all $k\in \N$. So, by applying \cite[Theorem 1.1]{Kr} to $u_{\eta_{|\bar B_1}}$, for any $\eps>0$ and $k\in\N$, we can find $u_\eta^\eps\in C^\infty_c(\R^n)$ such that
\begin{equation}\nonumber\left\{
\begin{aligned}
&(-\Delta)^s u_\eta^\eps=0\quad\mbox{in }\overline B_1,\\
&\|u_\eta-u_\eta^\eps\|_{C^k(B_1)}\leq\frac\eps 2.
\end{aligned}
\right.
\end{equation}
For $k>\alpha$, we can estimate
$$
\|u-u_\eta^\eps\|_{C^\alpha(B_1)}\leq \|u_\eta-u_\eta\|_{C^\alpha(B_1)}+\|u_\eta-u_\eta^\eps\|_{C^k(B_1)}\leq\eta+\frac \eps 2.
$$
We conclude by choosing $\eta<\frac \eps 2$.
\end{proof}
\noindent The following lemma provides instead that ``all functions are locally almost $s$-harmonic".
\begin{Lemma}{\bf (All functions are locally almost $s$-harmonic)}\label{all func}
Suppose that $u\in C^{\alpha}(\overline B_1)$ for some $\alpha>2s$. Then, for every $\eps>0$, there exists $u_\eps\in C^\alpha_c(\R^n)$ such that
\begin{equation}\nonumber\left\{
\begin{aligned}
|(-\Delta)^s u_\eps|&\leq\eps\quad&&\mbox{in }\overline B_1,\\
u_\eps&=u&&\mbox{in } B_1.
\end{aligned}
\right.
\end{equation}
\end{Lemma}

\begin{proof}
By suitably extending $u$ to the whole $\R^n$, we can assume that $u\in C^\alpha_c(\R^n)$.
Consider a cut-off function $\varphi\in C^\infty_c(\R^n)$ such that $\varphi=1$ in $B_2$, $\varphi=0$ in $\R^n\setminus B_{3}$, and $0<\varphi<1$ elsewhere. Let us fix $\eps>0$ and apply Theorem \ref{DSV} to $u_{|\bar B_3}\in C^\alpha(\overline B_3)$. Then, we can find a function $v_\eps\in C^\infty_c(\R^n)$  which is $s$-harmonic in $B_3$ and
\begin{align}\label{eps small}
\|u-v_\eps\|_{C^\alpha(B_3)}\leq\eps.
\end{align}
Now set $$u_\eps:=\varphi(u-v_\eps)+v_\eps\in C^\alpha_c(\R^n).$$
Note that $u_\eps=u$ in $B_2$. Moreover, we can compute its $s$-Laplacian using the Leibniz formula (see for instance \cite{BPSV}):
$$
(-\Delta)^s u_\eps=(u-v_\eps)(-\Delta)^s \varphi +\varphi\,(-\Delta)^s (u-v_\eps)- B(\varphi, u-v_\eps)+(-\Delta)^sv_\eps\qquad\mbox{in }\R^n.
$$
From the $s$-harmonicity of $v_\eps$ in $B_3$, we have 
\begin{align}\label{Leib}
(-\Delta)^s u_\eps=(u-v_\eps)(-\Delta)^s \varphi +\varphi(-\Delta)^s u- B(\varphi, u-v_\eps)\qquad\mbox{in }B_3.
\end{align}
Now, for $x\in \overline B_1$, we have 
\begin{align*}
B(\varphi, u-v_\eps)(x)=\, &c_{n,s}\int_{\R^n}\frac{(\varphi(x)-\varphi(y))((u-v_\eps)(x)-(u-v_\eps)(y))}{|x-y|^{n+2s}}\,dy\\
=\, &\varphi(x)\,c_{n,s}\left(\int_{\R^n}\frac{u(x)-u(y)}{|x-y|^{n+2s}}\,dy+\int_{\R^n}\frac{v_\eps(x)-v_\eps(y)}{|x-y|^{n+2s}}\,dy\right)-\\
&-c_{n,s}\int_{\R^n}\frac{\varphi(y)((u-v_\eps)(x)-(u-v_\eps)(y))}{|x-y|^{n+2s}}\,dy\\
=\,& \varphi(x)(-\Delta)^s u(x)-c_{n,s}\int_{\R^n}\frac{\varphi(y)((u-v_\eps)(x)-(u-v_\eps)(y))}{|x-y|^{n+2s}}\,dy,
\end{align*}
where we used that $v_\eps$ is $s$-harmonic in $\overline B_1$ and each integral is understood in the principal value sense.\\
Thus, by using \eqref{Leib} and setting $w_\eps:=u-v_\eps$, we have
$$
(-\Delta)^s u_\eps(x)=w_\eps(x)(-\Delta)^s \varphi(x)+c_{n,s}\int_{\R^n}\frac{\varphi(y)(w_\eps(x)-w_\eps(y))}{|x-y|^{n+2s}}\,dy \qquad\forall x\in \overline B_1.
$$
Since $\varphi=1$ in $B_2$ and vanishes out of $B_3$, that the last integral can be written as
\begin{align*}
\int_{\R^n}\frac{\varphi(y)(w_\eps(x)-w_\eps(y))}{|x-y|^{n+2s}}\,dy&=\int_{B_3}\frac{\varphi(y)(w_\eps(x)-w_\eps(y))}{|x-y|^{n+2s}}\,dy\\
&=\int_{B_{\frac12}(x)}\frac{w_\eps(x)-w_\eps(y)}{|x-y|^{n+2s}}\,dy+\int_{B_3\setminus B_{\frac12}(x)}\frac{\varphi(y)(w_\eps(x)-w_\eps(y))}{|x-y|^{n+2s}}\,dy.
\end{align*}
In order to estimate the first integral in the right hand side, we distinguish between the following cases. It is not restrictive to assume that $\alpha\in\R^+\setminus\N$ and $\alpha<2.$
\begin{itemize}
\item
{{$2s<\alpha<1$}}.\\
By using the H\"older continuity of $w_\eps$ and \eqref{eps small}, we have
$$
\left|\int_{B_{\frac12}(x)}\frac{w_\eps(x)-w_\eps(y)}{|x-y|^{n+2s}}\,dy\right|\leq[w_\eps]_{C^\alpha(B_3)}\int_{B_{\frac12}(x)}\frac{dy}{|x-y|^{n+2s-\alpha}}\leq \eps\, C,
$$
where $C=C(n,s,\alpha):=\int_{B_{\frac12}(x)}\frac{dy}{|x-y|^{n+2s-\alpha}}$ is a positive constant.\\
\item
{{ $2s\leq1<\alpha$}}.\\
This case is similar to the previous one.\\

\item
{{$1<2s<\alpha$}}.\\
By using the symmetry of $B_{\frac12}(x)$, the Lagrange mean value Theorem, and the fact that $\grad w_\eps\in C^{\alpha-1}(B_3,\R^n)$, we have
\begin{align*}
\left|\int_{B_{\frac12}(x)}\frac{w_\eps(x)-w_\eps(y)}{|x-y|^{n+2s}}\,dy\right|&=\left|\int_{B_{\frac12}(x)}\frac{w_\eps(x)-w_\eps(y)-\grad w_\eps(\xi)\cdot(x-y)}{|x-y|^{n+2s}}\,dy\right|\\
&\leq[\grad w_\eps]_{C^{\alpha-1}(B_3)}\int_{B_\frac12(x)}\frac{dy}{|x-y|^{n+2s-\alpha}}\\
&\leq \eps \, C.
\end{align*}
\end{itemize}
On the other hand, 
\begin{align*}
\left|\int_{B_3\setminus B_{\frac12}(x)}\frac{\varphi(y)(w_\eps(x)-w_\eps(y))}{|x-y|^{n+2s}}\,dy\right|&\leq\int_{B_3\setminus B_{\frac12}(x)}\frac{|w_\eps(x)-w_\eps(y)|}{|x-y|^{n+2s}}\,dy\\
&\leq 2\|w_\eps\|_{L^\infty(B_3)}2^{n+2s}|B_3|\\
&\leq C'\eps,
\end{align*}
for a constant $C'=C'(n,s)>0$.\\
Therefore, we can estimate the $s$-Laplacian of $u_\eps$ as
$$
|(-\Delta)^s u_\eps|\leq C''\eps\qquad\mbox{in }\overline B_1,
$$
for a constant $C''=C''\left(\|(-\Delta)^s \varphi\|_{L^\infty(B_1)}, C, C'\right)>0$ independent of $\eps$, providing the statement.
\end{proof}
As a consequence of Lemma \ref{all func}, we are able to construct a competitor $w_\eps\in \mathcal{W}_{u_0}^{2s,\infty}(\Om)$ with small Fractional Laplacian in $\R^n\setminus\Om$.
\begin{Corollary}{\bf (Improving boundary conditions)}\label{bc}
Suppose that $u_0\in C^{2s+\gamma}_c(\R^n)$, for some $\gamma>0$. Then, for every $\eps>0$, there exists $w_\eps\in C^{2s+\gamma}_c(\R^n)$ such that
\begin{equation}\nonumber\left\{
\begin{aligned}
|(-\Delta)^s w_\eps|&\leq\eps\quad&&\mbox{in }\R^n\setminus\Om,\\
w_\eps&=u_0&&\mbox{in } \R^n\setminus\Om.
\end{aligned}
\right.
\end{equation}
\end{Corollary}
\begin{proof}
{\textit{Step 1: case $\Om=B_1$}}.\\
 Consider the Kelvin Transform $\mathcal K:\R^n\setminus\{0\}\to\R^n\setminus\{0\}$ defined as
$$
\mathcal K(x):=\frac{x}{|x|^2}.
$$
Observe that $\mathcal K(B_1)=\R^n\setminus\overline B_1$ and $\mathcal K_{|\partial B_1}=\mathrm{id}_{|\partial B_1}$.
Now, consider the function $v:=(u_0)_\mathcal K: \R^n\to \R$ defined as 
$$v(x):=(u_0)_\mathcal{K}(x):=|x|^{2s-n}u_0(\mathcal K(x)), \qquad\forall x\in \R^n.$$
Since $u_0$ has compact support, the origin is not a singularity for $v$, which is then well defined and of class $C^{2s+\gamma}$ in $\R^n$. \\
Let us fix $\eps>0$. We can apply Lemma \ref{all func} to the restriction $v_{|B_1}\in C^{2s+\gamma}(\overline B_1)$, obtaining a function $v_\eps\in C^{2s+\gamma}_c(\R^n)$ such that
\begin{equation}\label{v eps}\left\{
\begin{aligned}
|(-\Delta)^s v_\eps|&\leq\eps\quad&&\mbox{in }\overline B_1,\\
v_\eps&=v&&\mbox{in } B_1.
\end{aligned}
\right.
\end{equation}
Now define $w_\eps:=(v_\eps)_\mathcal K$. Again, since $v_\eps$ has compact support, $w_\eps$ is well defined and $C^{2s+\gamma}$ in $\R^n$. Moreover, it is immediate to see that
$$
{w_\eps}=u_0\qquad\mbox{in }\R^n\setminus B_1.
$$
Finally, using in \cite[Prop. 1.22(ii)]{ADV} (which holds also for $C^{2s+\gamma}$-functions) and \eqref{v eps},  we can compute and estimate the Fractional Laplacian of $w_\eps$ for $x\in \R^n\setminus B_1$ as
$$
|(-\Delta)^sw_\eps(x)|=|x|^{-2s-n}|(-\Delta)^s v_\eps(\mathcal K(x))|\leq\eps.
$$
{\textit{Step 2: general case}}.\\ It is enough to prove the thesis in $\R^n\setminus B_r(x)$, for a fixed ball $B_r(x)\Subset\Om$. To this purpose, we can consider the generalised Kelvin Transform $\mathcal K_{r,x}:\R^n\setminus\{x\}\to\R^n\setminus\{x\}$, defined as
$$
\mathcal{K}_{r,x}(y)=r^2\frac{y-x}{|y-x|^2}+x
$$
and the transformation 
$$
u_{\mathcal K_{r,x}}(y)=|y-x|^{2s-n}u({\mathcal K_{r,x}}(y)),
$$
for a function $u\in C^{2s+\gamma}_c(\R^n)$. By repeating the argument of the previous case, the conclusion follows.
\end{proof}

\begin{Remark}
We recall that our non-trivial choice of $u_0$, i.e. $u_0\not\equiv0$ in $\R^n\setminus\Om$, makes the result of Corollary \ref{bc} not obvious. Moreover, one cannot expect to pass to the limit as $\eps\to0$ in the thesis. In other words, the result does not hold for $\eps=0$, since otherwise one would find an extension $w$ of ${u_0}_{|\R^n\setminus\Om}$ to the whole $\R^n$ which is $s$-harmonic in $\R^n\setminus\Om$. Since $s$-harmonic functions satisfy the unique continuation property on their $s$-harmonicity domain, $w$ is forced to vanish on $\R^n\setminus\Om$, contradicting the non-trivial choice of $u_0$.
\end{Remark}

Now we are ready to ensure the non-triviality of $f_\infty$ in $\Om$.

{\it Proof of Lemma \ref{non triv lemma}}
Let us fix $\eps>0$ and let $w_\eps\in C^{2s+\gamma}_c(\R^n)$ be the function given by Corollary \ref{bc}, corresponding to our data $u_0\in C^{2s+\gamma}_c(\R^n)$. By testing \eqref{EL for Ep} with $u_p-w_\eps\in\mathscr{L}^{2s,p}_{w,0}(\Om)$, we get
$$
\int_{\R^n} f_p(-\Delta)^su_p=\int_{\R^n}f_p(-\Delta)^sw_\eps.
$$
On the other hand, from the definition of $f_p$, we can write
\begin{align*}
\int_{\R^n} f_p(-\Delta)^su_p=e_p^{1-p}\int_{\R^n}w|(-\Delta)^su_p|^{p}= e_p.
\end{align*}
Thus,
$$
\int_{\R^n}f_p(-\Delta)^sw_\eps= e_p.
$$
Passing to the limit as $p\to\infty$, by recalling \eqref{limit ep}, \eqref{mu on Om}, and that $|\mu|(\R^n)\leq{1}$, we obtain
\begin{align*}
e_\infty&=\int_{\R^n}(-\Delta)^sw_\eps\,d\mu\\
&=\int_\Om f_\infty(-\Delta)^sw_\eps+\int_{\R^n\setminus\Om}(-\Delta)^sw_\eps\,d\mu\\
&\leq\int_\Om f_\infty(-\Delta)^sw_\eps+\|(-\Delta)^sw_\eps\|_{L^\infty(\R^n\setminus\Om)}|\mu|(\R^n)\\
&\leq\int_\Om f_\infty(-\Delta)^sw_\eps+\eps.
\end{align*}
Therefore, if $\eps<e_\infty$, we have
$$
\int_\Om f_\infty(-\Delta)^sw_\eps\geq e_\infty-\eps>0,
$$
providing that $f_\infty\not\equiv0$ in $\Om$. \qed

We are ready to derive the PDE in \eqref{PDE u infty}. Thanks to Lemma \ref{non triv lemma}, $f_\infty$ is a non-trivial analytic function in $\Om$, so the set $\Gamma:=f_\infty^{-1}(0)\subset\Om$ has null Lebesgue measure. Moreover, recalling that, for all $q\in (1,\infty)$, $u_p\rightharpoonup u_\infty$ weakly in $\mathscr{L}^{2s,q}_{w,u_0}(\Om)$, up to passing to a further subsequence, we have 
\begin{align}\label{weak conv Lapl}
(-\Delta)^su_p\rightharpoonup (-\Delta)^su_\infty\quad\mbox{ weakly in }L^q_{\mathrm{loc}}(\R^n).
\end{align}
In particular, the convergence in \eqref{weak conv Lapl} holds (weakly) in $L^q(\Om)$. Recalling \eqref{fp}, we can write
$$
|f_p|^\frac{1}{p-1}\mathrm{sgn}(f_p)=e_p^{-1}w^\frac{1}{p-1}(-\Delta)^su_p\qquad\mbox{in }\Om.
$$
If we restrict the above equation to a compact set $K\subset\Om\setminus\Gamma$ and pass to the limit as $p\to\infty$, thanks to \eqref{limit ep}, \eqref{unif conv}, and Lemma \ref{non triv lemma}, we infer
$$
\mathrm{sgn}(f_\infty)=e_\infty^{-1}(-\Delta)^su_\infty\qquad\mbox{in }K,
$$
which provides the validity of \eqref{PDE u infty}, since $|\Gamma|=0$.\\

Now we aim to derive \eqref{PDE u infty mu}. Having proved \eqref{PDE u infty}, it is enough to restrict our attention to $\R^n\setminus\overline\Om$.\\
First, let us show that $(-\Delta)^s u_\infty\in C^{\gamma}_{\mathrm{loc}}(\R^n\setminus\overline\Om)$\footnote{With the notation $C^{\gamma}_\mathrm{loc}$ we understand the space of locally H\"older continuous functions (see Section \ref{sec:not}).} and $(-\Delta)^s u_\infty(x)\to0$ as $|x|\to\infty$: Since $u_\infty=u_0$ in $\R^n\setminus\Om$, we can write 
$$
(-\Delta)^s u_\infty(x)=(-\Delta)^s u_0(x)+c_{n,s}\int_\Om\frac{u_0(y)-u_\infty(y)}{|x-y|^{n+2s}}dy\qquad\forall x\in\R^n\setminus\overline\Om.
$$ 
Set 
$$A_\infty(x):=\int_\Om\frac{u_0(y)-u_\infty(y)}{|x-y|^{n+2s}}dy\qquad\forall x\in\R^n\setminus\overline\Om.$$
Then, for any $\beta\in\N_0^n$, we can estimate the derivatives of order $\beta$ of $A_\infty$ as
$$
|D^\beta A_\infty(x)|\leq d(x,\partial\Om)^{-n-2s-|\beta|}(\|u_0\|_{L^1(\Om)}+\|u_\infty\|_{L^1(\Om)})\qquad\forall x\in\R^n\setminus\overline\Om.
$$
In particular, $A_\infty(x)\to0$ as $|x|\to\infty$. Hence, since $u_0\in C^{2s+\gamma}_c(\R^n)$, we have that  $(-\Delta)^s u_0(x)\to0$ as $|x|\to\infty$, thus the same holds for $(-\Delta)^s u_\infty$. Moreover, if $x\in\R^n\setminus\overline\Om$ and $r<\frac12d(x,\partial\Om)$, the triangle inequality implies
$$
|D^\beta A_\infty|\leq r^{-n-2s-|\beta|}(\|u_0\|_{L^1(\Om)}+\|u_\infty\|_{L^1(\Om)})\qquad\mbox{in }B_r(x).
$$ 
Thus, $A_\infty\in C^{\infty}(\R^n\setminus\overline\Om)$ and so, being $(-\Delta)^s u_0\in C^{\gamma}(\R^n)$ (see \cite[Prop. 2.6]{Si}), we have that $(-\Delta)^s u_\infty\in C^{\gamma}_{\mathrm{loc}}(\R^n\setminus\overline\Om)$.\\
As a consequence, we can show that $\mu$ has compact support in $\R^n$: Applying the previous argument to $(-\Delta)^su_p$, we infer $(-\Delta)^su_p(x)\to0$ as $|x|\to\infty$. More precisely, since $u_p\rightharpoonup u_\infty$ weakly in $\mathscr{L}^{2s,q}_{w,u_0}(\Om)$ for all $q>1$, the norm $\|u_p\|_{L^1(\Om)}$ is uniformly bounded w.r.t. $p$, thus the following uniform estimate holds: for every $\eps>0$ there exists $R>0$, independent of $p$, such that
$$
\|(-\Delta)^su_p\|_{L^\infty(\R^n\setminus B_R)}<\eps.
$$
Since $e_p$ is bounded from below by $e_1>0$, by taking $\eps<e_1$, we have that there exists $\delta\in(0,1)$, independent of $p$, so that
$$
\frac{|(-\Delta)^su_p|}{e_p}<1-\delta\qquad\mbox{in }\R^n\setminus B_R.
$$
Therefore,
$$
\left(\frac{|(-\Delta)^su_p|}{e_p}\right)^{p-1}\to0\quad\mbox{ uniformly in }\R^n\setminus B_R, \quad\mbox{as }p\to\infty.
$$
Recalling the definition of $f_p$ in \eqref{fp}, we get
\begin{align*}
|f_p|&= e_p^{1-p}w|(-\Delta)^su_p|^{p-1}\leq{\|w\|_{L^\infty(\R^n)}}\left(\frac{|(-\Delta)^su_p|}{e_p}\right)^{p-1}.
\end{align*}
Thus, 
$$f_p\to0\quad\mbox{ uniformly in }\R^n\setminus B_R, \quad\mbox{as }p\to\infty.$$
In particular, by lower semicontinuity of the mass, $\mu$ cannot concentrate out of $B_R$, providing that supp$(\mu)\Subset\R^n.$\\
Now we are ready to derive the PDE \eqref{PDE u infty mu}. As we have seen for $(-\Delta)^s u_\infty$, we can prove that $(-\Delta)^s u_p\in C^{\gamma}_{\mathrm{loc}}(\R^n\setminus\overline\Om)$ and for every ball $B\Subset\R^n\setminus\overline\Om$, there exists $C=C(B,n,s)$ such that
$$
\|(-\Delta)^s u_p\|_{C^{\gamma}(B)}\leq \|(-\Delta)^s u_0\|_{C^\gamma(B)}+C(\|u_0\|_{L^1(\Om)}+\|u_p\|_{L^1(\Om)}).
$$
Thus, the uniform bound of $\|u_p\|_{L^1(\Om)}$ and \eqref{weak conv Lapl} ensure 
\begin{align}\label{unif Lapl}
(-\Delta)^s u_p\to(-\Delta)^s u_\infty\quad\mbox{locally uniformly in }\R^n\setminus\overline\Om\quad\mbox{as }p\to\infty.
\end{align}
By writing
$$
|f_p|^\frac{1}{p-1}=e_p^{-1}w^\frac{1}{p-1}{|(-\Delta)^s u_p|},
$$
and, by recalling that $e_p\to e_\infty$, we have
$$
|f_p|^\frac{1}{p-1}\to e_\infty^{-1}{|(-\Delta)^s u_\infty|}\quad\mbox{locally uniformly in }\R^n\setminus\overline\Om\quad\mbox{as }p\to\infty.
$$
In particular, we observe that if $x\in\R^n\setminus\overline\Om$ is such that ${(-\Delta)^s u_\infty(x)}\neq0$, then, in a neighborhood of $x$, by continuity, $(-\Delta)^su_p$ and $(-\Delta)^su_\infty$ have (equal) constant sign. Therefore, since $\mathrm{sgn}(f_p)=\mathrm{sgn}(-\Delta)^su_p$, the uniform convergence in \eqref{unif Lapl} implies
\begin{align}\label{sgn conv}
\mathrm{sgn}(f_p)\to\mathrm{sgn}(-\Delta)^s u_\infty\quad\mbox{locally uniformly in }(\R^n\setminus\overline\Om)\cap\{(-\Delta)^s u_\infty\neq0\}\quad\mbox{as }p\to\infty.
\end{align}
Now, from the $L^1$ bound of $f_p$ in $\R^n$, we can also assume that there exists a measure $\nu\in\mathcal{M}^+(\R^n)$, with $|\nu|(\R^n)\leq{1}$ and $|\mu|<<\nu$, such that
$$
|f_p|\stackrel{*}{\rightharpoonup}\nu\quad\mbox{in }\mathcal{M}(\R^n).
$$
So, by using \eqref{sgn conv}, we get
$$
f_p=|f_p|\,\mathrm{sgn}f_p\stackrel{*}{\rightharpoonup}\nu \,\mathrm{sgn}(-\Delta)^s u_\infty\quad\mbox{in  }(\R^n\setminus\overline\Om)\cap\{(-\Delta)^s u_\infty\neq0\}.
$$
On the other hand,
$$
f_p\stackrel{*}{\rightharpoonup}\mu=|\mu|\frac{d\mu}{d|\mu|}\quad\mbox{in }\R^n.
$$
Therefore, by uniqueness of the polar decomposition of measures, we have
\begin{align}\label{polar dec}
\nu=|\mu|\quad\mbox{and }\quad \frac{d\mu}{d|\mu|}=\mathrm{sgn}(-\Delta)^s u_\infty\quad |\mu|-\mbox{a.e. in  }(\R^n\setminus\overline\Om)\cap\{(-\Delta)^s u_\infty\neq0\}.
\end{align}
Now, we claim that supp$(|\mu|\mres\R^n\setminus\overline\Om)\subset\{x\in\R^n\setminus\overline\Om:|(-\Delta)^s u_\infty(x)|=e_\infty\}$. Indeed, suppose by contradiction that $|(-\Delta)^s u_\infty(x)|<e_\infty$ for some $x\in$ supp$(|\mu|\mres\R^n\setminus\overline\Om)$. Then, by \eqref{unif Lapl}, there exists $\eps\in(0,1)$, independent of $p$, and a radius $r>0$, such that
$$
\frac{|(-\Delta)^s u_p|}{e_p}\leq 1-\eps\quad\mbox{in } B_r(x).
$$
But then
$$
\left(\frac{|(-\Delta)^s u_p|}{e_p}\right)^{p-1}\to0\quad\mbox{uniformly in } B_r(x)\quad\mbox{as }p\to\infty,
$$
implying that
$$
|f_p|\to0\quad\mbox{uniformly in } B_r(x)\quad\mbox{as }p\to\infty.
$$
So, we have $|\mu|(B_r(x))=0$, leading to a contradiction. Therefore, taking into account \eqref{polar dec}, for every $x\in \mathrm{supp}(|\mu|\mres\R^n\setminus\overline\Om)$, we conclude that
$$
(-\Delta)^s u_\infty(x)=|(-\Delta)^s u_\infty(x)|\,\mathrm{sgn}(-\Delta)^s u_\infty(x)=e_\infty \frac{d\mu}{d|\mu|}.
$$

\subsection{Uniqueness}\label{subs: uniqueness}
We are left with proving the uniqueness of the minimiser $u_\infty$ in $\mathcal{W}^{2s,\infty}_{u_0}(\Om)$. The key step is to show that every minimiser of $E_\infty$ satisfies a fractional PDE in $\Om$ of the same structure as \eqref{PDE u infty}. The main idea is to introduce a penalisation term into the Gamma Convergence argument of Section \ref{subs:existence}, forcing the minimiser of the $L^p$-problem to convergence to a preselected minimiser of $E_\infty$. More explicitly, we prove the following result.
\begin{Lemma}{\bf (Necessity of the PDE)}\label{necessity of PDE}
Let $u_0$ and $\Om$ be as in Theorem \ref{main thm}. Suppose that $u\in\mathcal{W}^{2s,\infty}_{u_0}(\Om)$ is a minimiser of the functional $E_\infty$ defined in \eqref{E infty}. Then there exists an analytic function $f\in L^1(\Om)\setminus\{0\}$ such that
\begin{align}\label{PDE u}
(-\Delta)^su=e_\infty\mathrm{sgn}f\qquad\mbox{a.e.  in }\,\Om .
\end{align}

\end{Lemma}
\begin{proof}
Let $u\in\mathcal{W}^{2s,\infty}_{u_0}(\Om)$ be a minimiser of $E_\infty$. For any $p\geq2$, we consider the auxiliary functional
$$
A_p(v):=E_p(v)+\frac{1}{2}\fint_\Om|v-u|^2,\,\qquad v\in \mathscr{L}^{2s,p}_{w,u_0}(\Om),
$$
where the functional $E_p$ is defined in \eqref{Ep}.\\
Using the Direct Method, as we did in Section \ref{subs:existence}, it is not difficult to see that $A_p$ attains a minimum point. Let $v_p\in \mathscr{L}^{2s,p}_{w,u_0}(\Om)$ be a minimiser of $A_p$ and $u_p$ be a minimiser of $E_p$ in the same space. We have
$$
e_p=E_p(u_p)\leq E_p(v_p)\leq A_p(v_p)\leq A_p(u)=E_p(u)\leq E_\infty(u)=e_\infty.
$$
Since $e_p\to e_\infty$, we have 
\begin{align}\label{Ap limit}
\lim_{p\to\infty}E_p(v_p)=\lim_{p\to\infty}A_p(v_p)=e_\infty.
\end{align}
 Now, if $(v_p)_p$ is a sequence of minimisers of $A_p$, by the Calderon-Zygmund estimates, we can prove as in Section \ref{subs:existence} that $(v_p)_p$ is bounded in $\mathscr{L}^{2s,q}_{w,u_0}(\Om)$ for any fixed $q>1$. So, up to considering a subsequence, we have $v_{p}\rightharpoonup v_\infty$ weakly in $ \mathscr{L}^{2s,q}_{w,u_0}(\Om)$ for some $v_\infty\in\mathscr{L}^{2s,q}_{w,u_0}(\Om)$.\\
Let us prove that $v_\infty=u$ in $\R^n$: Passing to the limit as $p\to\infty$ in the following inequality
$$
e_{p}+\frac{1}{2}\fint_\Om|v_{p}-u|^2\,\leq A_{p}(v_{p}),
$$
by \eqref{Ap limit}, we get
$$
e_\infty+\frac{1}{2}\fint_\Om|v_\infty-u|^2\,\leq e_\infty,
$$
which necessarily implies $v_\infty=u$ in $\Om$, and therefore in $\R^n$, since they share the same data in the complement of $\Om$.\\
Now, denote $a_p:=E_p(v_p)$ and write down the Euler-Lagrange equation satisfied by $v_p$:
\begin{align}\label{EL for vp}
(-\Delta)^s g_p+v_p-u=0\qquad\mbox{in }\left(\mathscr{L}^{2s,p}_{w,0}(\Om)\right)^*,
\end{align}
 where $g_p$ is defined by 
\begin{align}\label{gp}
g_p:=a_p^{1-p}w|(-\Delta)^sv_p|^{p-2}(-\Delta)^sv_p\qquad\mbox{a.e. in }\R^n.
\end{align}
 We can prove that $g_p$ is well defined and uniformly bounded in $L^1(\R^n)$ as showed for $f_p$ in Section \ref{subs:existence}.
Moreover, by solving the Dirichlet problem
\begin{equation}\nonumber\left\{
\begin{aligned}
(-\Delta)^s h_p&=v_p-u\quad&&\mbox{in }\Om,\\
h_p&=0&&\mbox{in } \R^n\setminus\Om,
\end{aligned}
\right.
\end{equation}
and applying Thorem \ref{CZ thm}, the unique solution $h_p$ is such that
$$
\|h_p\|_{W^{s,q}(\R^n)}\leq C\|v_p-u\|_{L^q(\Om)}.
$$
In particular, the function $g_p+h_p\in L^1(\R^n)$ satisfies
$$
(-\Delta)^s (g_p+h_p)=0\qquad\mbox{in }\left(\mathscr{L}^{2s,p}_{w,0}(\Om)\right)^*.
$$
Hence, $g_p+h_p$ is a very weak solution to the Fractional Laplace problem in $\Om$ (w.r.t. its own Dirichlet boundary data). By repeating the same argument of Section \ref{subs: PDE derivation}, elliptic estimates imply that $g_p+h_p$ is actually a classical solution. In particular, for every ball $B\Subset\Om$, there exists $\sigma\in(0,1)$ such that
$$
\|g_p+h_p\|_{C^{0,\sigma}(B)}\leq C\left\|g_p+h_p\|_{L^1(\R^n)}\leq C(\|g_p\|_{L^1(\R^n)}+\|v_p-u\|_{L^q(\Om)}\right)\leq C,
$$
for some constant $C=C(\sigma,n,s,B)>0$ independent of $p$ (we are using the $L^1$-bound of $g_p$ and that $v_p\rightharpoonup u$ weakly in $ \mathscr{L}^{2s,q}_{w,u_0}(\Om)$ for any $q>1$). 
Since $h_p\to0$ in $W^{s,q}(\R^n)$ as $p\to\infty$, we have $g_p\to g_\infty$ locally uniformly as $p\to\infty$, for some $g_\infty\in L^1(\Om)$, up to passing to a subsequence. One can guarantee the analiticity of $g_\infty$ as we did for $f_\infty$ in Section \ref{subs: PDE derivation}, since it is the restriction to $\Om$ of an $s$-harmonic measure $\nu\in \mathcal M(\R^n)$ in $\Om$ (the limit measure of $g_p$). Furthermore, as in Lemma \ref{non triv lemma}, we can prove that $g_\infty\not\equiv0$, since the same argument holds even if $(-\Delta)^s g_p\rightharpoonup0$ in the sense of distributions only (which is guaranteed by \eqref{EL for vp}). Finally, since $a_p\to e_\infty$ as $p\to\infty$, we can pass to the limit in \eqref{gp} exactly as we did for $f_p$, obtaining \eqref{PDE u} with $f=g_\infty$.
\end{proof}
We conclude the section with showing the uniqueness of minimisers of $E_\infty$. Consider two minimisers $u_1,u_2\in\mathcal{W}_{u_0}^{2s,\infty}(\Om)$. Then, thanks to Lemma \ref{necessity of PDE} they both satisfy 
\begin{align}\label{u12}
\left|(-\Delta)^s{u_1}\right|=\left|(-\Delta)^s{u_2}\right|=e_\infty\qquad\mbox{ a.e. in }\Om.
\end{align}
Now, notice that the average $\frac{u_1+u_2}{2}\in\mathcal{W}_{u_0}^{2s,\infty}(\Om)$. Hence, we have
$$
\left\|(-\Delta)^s\left(\frac{u_1+u_2}{2}\right)\right\|_{L^\infty(\R^n)}\geq e_\infty.
$$
On the other hand, the triangle inequality gives
$$
\left\|(-\Delta)^s\left(\frac{u_1+u_2}{2}\right)\right\|_{L^\infty(\R^n)}\leq\frac{1}{2}\|(-\Delta)^s{u_1}\|_{L^\infty(\R^n)}+\frac{1}{2}\|(-\Delta)^s{u_2}\|_{L^\infty(\R^n)}=e_\infty.
$$
Therefore, we obtain
$$
\left\|(-\Delta)^s\left(\frac{u_1+u_2}{2}\right)\right\|_{L^\infty(\R^n)}= e_\infty,
$$
giving that $\frac{u_1+u_2}{2}$ is a minimiser as well.
By Lemma \eqref{necessity of PDE}, we have
\begin{align}\label{u ave}
\left|(-\Delta)^s\left(\frac{u_1+u_2}{2}\right)\right|=e_\infty\qquad\mbox{ a.e. in } \Om.
\end{align}
Putting together \eqref{u12} and \eqref{u ave}, the triangle inequality implies $$(-\Delta)^s{u_1}=(-\Delta)^s{u_2}\qquad\mbox{a.e. in }\Om.$$
In particular, the function $v=u_1-u_2$ solves
\begin{equation}\nonumber\left\{
\begin{aligned}
(-\Delta)^sv&=0\quad&&\mbox{in }\Om,\\
v&=0&&\mbox{in }\R^n\setminus\Om.
\end{aligned}
\right.
\end{equation}
By uniqueness of weak solutions to the Dirichlet problem for the Fractional Laplacian, we conclude that $u_1=u_2$ in $\R^n$.

\begin{Remark}[Pointwise representative of $(-\Delta)^su_\infty$]
We conclude this section by observing that $(-\Delta)^su_\infty\in L^\infty(\R^n)$ holds actually in the pointwise sense, not only in the weak one. Indeed, by arguing as in Remark \ref{rem on class}, one can easily infer that $u_\infty-u_0\in W^{2s,p}(\R^n)$ for all $p\in [2,\infty)$. So that $u_\infty\in W^{2s,p}(\R^n)$ as well, since $u_0\in C^{2s+\gamma}_c(\R^n)\subset W^{2s,p}(\R^n)$. In particular, $(-\Delta)^su_\infty\in L^2(\R^n)$ in the usual sense, so it possesses a pointwise representative.
\end{Remark}

\section{More general supremands}\label{sec:gen}
In this last section, we point out that Theorem \ref{main thm} can be generalised to supremands of the form $F(x,(-\Delta)^s u(x))$, where $F:\R^n\times\R\to\R$ is a Carathéodory function such that, for almost every $x\in\R^n$, $F(x,0)=0$, $F(x,\cdot)$ is of class $C^1$, and $|F(x,\cdot)|$ is strictly convex. Assume also:
\begin{equation}\label{ass F}
\begin{aligned}
& \exists c>0:\quad c\leq F_\xi(x,\xi)\leq\frac1c \qquad\mbox{a.e. }x\in\R^n,\, \forall\xi\in\R.
\end{aligned}
\end{equation}
More explicitely, by suitably modifying the proof in Section \ref{sec:proofs} according to the assumptions on $F$ (in the same spirit of \cite{KM}), one obtains the following result for the functional
$$
E_\infty(u):=\|F\left(\cdot,(-\Delta)^s u\right)\|_{L^\infty(\R^n)}.
$$

\begin{Theorem}\label{main thm bis}
Fix $s\in (0,1)$ and $n\in\N$, $n>2s$. Let $\Om\subset\R^n$ be an open bounded set and assume that $u_0\in C^{2s+\gamma}_c(\R^n)$, for some $\gamma>0$, with $u_0\not\equiv0$ in $\R^n\setminus\Om$.
Then the problem
$$
e_\infty:=\inf_{\mathcal{W}_{u_0}^{2s,\infty}(\Om)}E_\infty
$$
admits a unique solution $u_\infty\in\mathcal{W}_{u_0}^{2s,\infty}(\Om)$. \\
In particular, $(-\Delta)^s u_\infty\in C^\gamma_{\mathrm{loc}}(\R^n\setminus\overline\Om)$ and $(-\Delta)^su_\infty(x)\to0$ as $|x|\to+\infty$.\\
Moreover, a system of PDEs can be derived as a necessary and sufficient condition for the minimality of $u_\infty$. Explicitely, there exists a measure $\mu\in\mathcal{M}(\R^n)$, $\mu\neq0$, with compact support and $|\mu|(\R^n)\leq\frac1c$, such that $\mu$ is $s$-harmonic in $\Om$ and 
\begin{align}\label{PDE u infty bis 1}
F(\cdot,(-\Delta)^su_\infty)=e_\infty\frac{d\mu}{d|\mu|}\qquad\mbox{in supp}|\mu|\setminus\partial\Om.
\end{align}
The identity above is understood between $L^\infty$-functions on supp$|\mu|\setminus\partial\Om$.\\
Moreover, the restriction $\mu\mres\Om$ is absolutely continuous w.r.t. the Lebesgue measure on $\Om$, i.e. $\mu\mres\Om=f_\infty\mathscr{L}^n\mres\Om$, for some function $f_\infty\in L^1(\Om)\setminus\{0\}$, which is real analytic in $\Om$. In particular, there holds
\begin{align}\label{PDE u infty bis 2}
F(\cdot,(-\Delta)^su_\infty)=e_\infty\mathrm{sgn}f_\infty\qquad\mbox{a.e.  in }\,\Om .
\end{align}
\end{Theorem}

\begin{Remark}
The assumptions on $F$ may seem restrictive at first, however, due to the $L^\infty$-structure of the problem, this shares the same solutions with any other problem associated to a reparametrised supremand of the form $g\circ F$, for any lower semicontinuous, strictly increasing function $g:\R\to\R$ (compare \cite{KM,CKM}). Thus, the result of Theorem \ref{main thm bis} actually applies for a wider class of supremands. 
\end{Remark}

\textit{Proof of Theorem \ref{main thm bis}}:\\
\noindent
\textit{i) Strict positivity of $e_\infty$}. Assume by contradiction that $e_\infty=0$. Then, there exists a sequence $(u_k)_k\subset\mathcal{W}_{u_0}^{2s,\infty}(\Om)$ such that $F(\cdot,(-\Delta)^su_k)\to0$ uniformly in $\R^n$. From the assumptions on $F$, we must have
$$
(-\Delta)^su_k\to0\qquad\mbox{uniformly in }\R^n.
$$
Thus, we conclude as in \ref{subsec strict}.\\

\noindent
\textit{ii) Existence.} For $p>1$ and $w$ as in Section \ref{subs:existence}, minimise the functional
\begin{align}\label{Ep bis}
E_p(u):=\left(\int_{\R^n}\left|F(x,(-\Delta)^su(x))\right|^pw(x)dx\right)^{\frac{1}{p}}
\end{align}
on $\mathscr{L}^{2s,p}_{w,u_0}(\Om)$.\\
By integration of \eqref{ass F}, we have
\begin{align}\label{F xi}
c|\xi|\leq F(x,\xi)\leq\frac{|\xi|}c \qquad\mbox{a.e. }x\in\R^n,\, \forall\xi\in\R.
\end{align}
This ensures that $E_p$ is well defined on $\mathscr{L}^{2s,p}_{w,u_0}(\Om)$ and (weakly) coercive. In addition, by convexity of $|F(x, \cdot)|$, $E_p$ is also weakly lower semicontinuous in $\mathscr{L}^{2s,p}_{w}(\R^n)$. Thus, Direct Methods provide the existence of a minimiser $u_p$. The Gamma Convergence argument applies as in Section \ref{subs:existence}, giving the existence of a minimiser $u_\infty\in\mathcal{W}_{u_0}^{2s,\infty}(\Om)$ of $E_\infty$. Up to passing to a subsequence, as $p\to\infty$, we have 
\begin{equation}\nonumber
\left\{
\begin{aligned}
& u_p\rightharpoonup u_\infty \mbox{ weakly in }\mathscr{L}^{2s,q}_{w,u_0}(\Om),\\
& e_p:=E_p(u_p)\to E_\infty(u_\infty)=e_\infty.
\end{aligned}
\right.
\end{equation}
\textit{iii) PDE derivation.} The Euler-Lagrange equation satisfied by $u_p$ is
\begin{align}\label{EL for Ep bis}
\int_{\R^n}f_p(x)(-\Delta)^sv(x)\, dx=0,\quad \forall v\in \mathscr{L}^{2s,p}_{w,0}(\Om),
\end{align}
where $f_p:\R^n\to\R$ is defined by
\begin{align*}\label{fp}
f_p:=e_p^{1-p}w|F(\cdot, (-\Delta)^su_p)|^{p-2}F(\cdot, (-\Delta)^su_p) F_\xi(\cdot, (-\Delta)^su_p).
\end{align*}
By using \eqref{F xi}, we obtain
$$
\int_{\R^n}|f_p|\leq\frac 1c,
$$
thus $(f_p)_p$ is bounded in $L^1(\R^n)$. Moreover, $f_p$ is a very weak $s$-harmonic function on $\Om$. So, up to extracting a subsequence, it converges in the sense of measure to an $s$-harmonic measure $\mu\in\mathcal M(\R^n)$ with $|\mu|(\R^n)\leq \frac 1c$.\\
 Moreover, by elliptic estimates, as in Section \ref{subs: PDE derivation}, we have that $f_p$ converges uniformly in $\Om$ to a real analytic function $f_\infty\in L^1(\Om)$. Of course, $\mu\mres\Om=f_\infty\mathscr L^n\mres\Om$.\\
Let us show the non-triviality of $f_\infty$. Let $w_\eps$ be as in Corollary \ref{bc}. By testing \eqref{EL for Ep bis} with $u_p-w_\eps$, using that sgn$F(\cdot, (-\Delta)^su_p)=$ sgn$(-\Delta)^su_p$ and that $F_\xi(\cdot,\xi)|\xi|\geq c^2|F(x,\xi)|$, we have
$$
\int_{\R^n}f_p (-\Delta)^sw_\eps=\int_{\R^n}f_p  (-\Delta)^su_p\geq c^2e_p^{1-p}\int_{\R^n}w |F(\cdot, (-\Delta)^su_p)|^p=c^2 e_p.
$$
So that, passing to the limit as $p\to \infty$ and using Corollary \ref{bc}, we obtain
\begin{align*}
c^2e_\infty=\int_{\R^n}(-\Delta)^sw_\eps\,d\mu\leq\int_\Om f_\infty(-\Delta)^sw_\eps+\frac\eps c,
\end{align*}
providing $f_\infty\not\equiv0$ in $\Om$, for $\eps$ small enough.\\
To derive \eqref{PDE u infty bis 2}, we write
$$
|f_p|^\frac{1}{p-1}\mathrm{sgn}(f_p)=e_p^{-1}w^\frac{1}{p-1}F(\cdot,(-\Delta)^su_p)F_\xi(\cdot,(-\Delta)^su_p)^{\frac 1{p-1}}\qquad\mbox{in }\Om.
$$
Thanks to \eqref{F xi}, $F_\xi(\cdot,(-\Delta)^su_p)$ is bounded from above and below, so we can pass to the limit in the previous equation obtaining \eqref{PDE u infty bis 2}.\\
We can argue similarly to Section \ref{subs:existence} to derive \eqref{PDE u infty bis 1}.\\

\noindent
\textit{iv) Uniqueness.} As in Section \ref{subs: uniqueness}, we can prove that any minimiser $u\in \mathcal{W}_{u_0}^{2s,\infty}(\Om)$ must satisfy
\begin{align*}
F(\cdot,(-\Delta)^su)=e_\infty\mathrm{sgn}f\qquad\mbox{a.e.  in }\,\Om,
\end{align*}
for some $f\in L^1(\Om)$ real analytic. \\
Now, let $u_1,u_2\in \mathcal{W}_{u_0}^{2s,\infty}(\Om)$ be two minimisers. Then
$$
|F(\cdot,(-\Delta)^su_1)|=|F(\cdot,(-\Delta)^su_2)|=e_\infty\qquad \mbox{a.e. in } \Om.
$$
By reasoning as in Section \ref{subs: uniqueness}, the strict convexity of $|F(x,\cdot)|$ allows to conclude
$$(-\Delta)^s{u_1}=(-\Delta)^s{u_2}\qquad\mbox{a.e. in }\Om.$$
So that $u_1=u_2$ in $\R^n$ by the uniqueness of the fractional Dirichlet problem. \qed
\\

\textsc{Acknowledgements:} S.C. is grateful to Lorenzo Brasco, Alessandro Carbotti, Serena Dipierro and Enrico Valdinoci for useful discussion and suggestions.  S.C and R.M. acknowledge partial financial support through the EPSRC grant EP/X017206/1. S.C acknowledges partial financial support through the EPSRC grant EP/X017109/1. 
S.C. is a member of Gruppo Nazionale per
l’Analisi Matematica, la Probabilità e le loro Applicazioni (GNAMPA) of the Istituto Nazionale
di Alta Matematica (INdAM) of Italy.\\

\textsc{Data availability}: Our manuscript has no associated data.\\

\textsc{Conflict of interest}: The authors have no Conflict of interest to declare for this article.


\vspace{3mm}
\textsc{Simone Carano}\\
Department of Mathematical Sciences, University of Bath\\
Bath BA2 7AY, UK\\
E-mail: sc3705@bath.ac.uk\\


\textsc{Roger Moser}\\
Department of Mathematical Sciences, University of Bath\\
Bath BA2 7AY, UK\\
E-mail: r.moser@bath.ac.uk

\end{document}